\documentclass[
	english
]{scrartcl}

\usepackage[T1]{fontenc}
\usepackage[utf8]{inputenc}

\usepackage{amsmath,amsfonts,amsthm,amssymb}
\usepackage[colorlinks,linkcolor=blue,citecolor=blue]{hyperref}

\usepackage{preamble}

\usepackage{enumitem}
\setenumerate{label=\textup{(\roman*)}}

\numberwithin{equation}{section} 
\usepackage[capitalise,nameinlink]{cleveref}

\usepackage{lmodern}

\newif\ifbiber
\bibertrue

\ifbiber
\usepackage[
	style=authoryear-comp,
	sorting=nyt,
	url=true,
	doi=true,                
	eprint=true,
	isbn=false,
	maxbibnames=99,          
	maxcitenames=2,
	uniquelist=minyear,
	dashed=false,            
	sortcites=false,          
	backend=biber,           
	safeinputenc,            
]{biblatex}

\DeclareCiteCommand{\cite}{%
	\ifbibmacroundef{cite:init}{}{\usebibmacro{cite:init}}\usebibmacro{prenote}%
}{%
	\usebibmacro{citeindex}%
	\printtext[bibhyperref]{\usebibmacro{cite}}%
}{%
	\ifbibmacroundef{cite:init}{\multicitedelim}{}%
}{%
	\usebibmacro{postnote}%
}%
\DeclareCiteCommand{\parencite}[\mkbibbrackets]{%
	\ifbibmacroundef{cite:init}{}{\usebibmacro{cite:init}}\usebibmacro{prenote}%
}{%
	\usebibmacro{citeindex}%
	\printtext[bibhyperref]{\usebibmacro{cite}}%
}{%
	\ifbibmacroundef{cite:init}{\multicitedelim}{}%
}{%
	\usebibmacro{postnote}%
}%

\let\cite\parencite

\DeclareNameAlias{sortname}{family-given}
\DeclareNameAlias{default}{family-given}

\else

\usepackage{doi,natbib}

\fi

\RequirePackage{mathtools}

\providecommand\given{\nonscript\;\delimsize|\nonscript\;\mathopen{}}
\DeclarePairedDelimiterX\set[1]\{\}{#1}
\DeclarePairedDelimiterX\seq[1](){#1}

\DeclarePairedDelimiterX\dual[2]{\langle}{\rangle}{#1,#2}

\DeclarePairedDelimiter\abs{\lvert}{\rvert}
\DeclarePairedDelimiter\norm{\lVert}{\rVert}

\DeclarePairedDelimiter\parens()
\DeclarePairedDelimiter\bracks[]

\newcommand\N{\mathbb{N}}
\newcommand\R{\mathbb{R}}

\renewcommand\d{\mathrm{d}}

\newcommand{\dualspace}{^\star}

\newcommand{\graph}{\operatorname{graph}}

\newcommand\dom{\operatorname{dom}}
\newcommand\interior{\operatorname{int}}

\newcommand\orcid[1]{%
	\hspace{.25em}%
	\href{http://orcid.org/#1}{%
		\protect\includegraphics[height=1em]{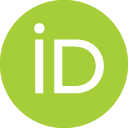}%
	}%
	\hspace{.25em}%
}

\newtheorem{theorem}{Theorem}[section]
\newtheorem{lemma}[theorem]{Lemma}
\newtheorem{remark}[theorem]{Remark}
\newtheorem{example}[theorem]{Example}

\KOMAoptions{headsepline=true}

\begin{document}
\title{A simple proof of the Baillon--Haddad theorem on open subsets of Hilbert spaces%
\footnote{
This research was supported by the German Research Foundation (DFG) under grant numbers WA~3626/3-2 and WA 3636/4-2
 within the priority program ``Non-smooth and Complementarity-based Distributed Parameter
Systems: Simulation and Hierarchical Optimization'' (SPP 1962).}
}

\author{%
	Daniel Wachsmuth%
	\footnote{
		Institut f\"ur Mathematik,
		Universit\"at W\"urzburg,
		97074 W\"urzburg, Germany, \email{daniel.wachsmuth@mathematik.uni-wuerzburg.de}
	}
	\and
	Gerd Wachsmuth%
	\footnote{%
		Brandenburgische Technische Universität Cottbus-Senftenberg,
		Institute of Mathematics,
		03046 Cott\-bus, Germany,
		\url{https://www.b-tu.de/fg-optimale-steuerung},
		\email{gerd.wachsmuth@b-tu.de},
		\orcid{0000-0002-3098-1503}%
	}
}
\publishers{}
\maketitle

\begin{abstract}
	We give a simple proof of the Baillon--Haddad theorem
	for convex functions defined
	on open and convex subsets of Hilbert spaces.
	We also state some generalizations and limitations.
	In particular, we discuss equivalent characterizations of the
	Lipschitz continuity of the derivative of convex functions
	on open and convex subsets of Banach spaces.
\end{abstract}

\begin{keywords}
	Baillon--Haddad theorem,
	cocoercivity,
	strong smoothness
\end{keywords}

\begin{msc}
	\mscLink{26B25},
	\mscLink{47H05},
	\mscLink{47N10},
	\mscLink{49J50}
\end{msc}

\section{Introduction}
\label{sec:intro}
A very important result in convex analysis is the
Baillon--Haddad theorem
which states
that the derivative $f'$ is $\frac1L$-cocoercive
whenever $f \colon X \to \R$
is convex and differentiable with $L$-Lipschitz continuous derivative,
see
\cite[Corollaire~10]{BaillonHaddad1977}.
Here, $X$ is a (real) Banach space.
In \cite[Theorem~3.1]{PerezArosVilches2019}
it was shown that this remains true
if $f$ is defined on an open and convex subset of a (real) Hilbert space.
The corresponding proof is quite involved,
since it uses generalized second-order derivatives
and a reduction to the finite-dimensional situation.
We also refer to \cite{BauschkeCombettes2009} for further comments and references.
We give a short and direct proof,
see \cref{sec:proof}.

It is well known that the $L$-Lipschitz continuity of the derivative of $f$ is
equivalent to a number of important properties of $f$ and its convex conjugate.
In \cref{sec:gen_and_lim}, we investigate which of these equivalences remain valid
if $f$ is defined on an open and convex subset of a (real) Banach or Hilbert space.

\section{The Baillon--Haddad theorem on open subsets of Hilbert spaces}
\label{sec:proof}
We start by a characterization of differentiable functions with Lipschitz derivatives.
\begin{lemma}
	\label{lem:lipschitz_derivative}
	Let $O \subset H$ be an open and convex subset of the Hilbert space $H$.
	Suppose that $f \colon O \to \R$ is (Gâteaux) differentiable.
	For $L \ge 0$, the following are equivalent.
	\begin{enumerate}
		\item
			\label{lem:lipschitz_derivative:1}
			The derivative $f' \colon O \to H\dualspace$ is $L$-Lipschitz,
			i.e.,
			\begin{equation}
				\label{eq:L_lipschitz}
				\norm{
					f'(y) - f'(x)
				}_{H\dualspace}
				\le
				L \norm{
					y - x
				}_H
				\qquad\forall x,y \in O
				.
			\end{equation}
		\item
			\label{lem:lipschitz_derivative:2}
			Both $f'$ and $-f'$ satisfy a one-sided Lipschitz estimate with constant $L$,
			i.e.,
			\begin{equation}
				\label{eq:two_sided_lipschitz}
				\abs{
					\dual{f'(y) - f'(x)}{y - x}_H
				}
				\le
				L \norm{y - x}_H^2
				\qquad\forall x,y \in O.
			\end{equation}
		\item
			\label{lem:lipschitz_derivative:3}
			The function $f$
			has a first order Taylor expansion with remainder $\frac L2 \norm{\cdot}_H^2$, i.e.,
			\begin{equation}
				\label{eq:taylor}
				\abs[\big]{
					f(y) - f(x) - \dual{f'(x)}{y - x}_H
				}
				\le
				\frac L2 \norm{y - x}_H^2
				\qquad\forall x,y \in O.
			\end{equation}
	\end{enumerate}
\end{lemma}
\begin{proof}
	The implication ``\ref{lem:lipschitz_derivative:1}$\Rightarrow$\ref{lem:lipschitz_derivative:2}''
	is a simple application of the Cauchy--Schwarz inequality.

	To prove ``\ref{lem:lipschitz_derivative:2}$\Rightarrow$\ref{lem:lipschitz_derivative:3}'',
	we employ the fundamental theorem of calculus
	and get
	\begin{align*}
		\abs[\big]{
			f(y) - f(x) - \dual{f'(x)}{y - x}_H
		}
		&\le
		\abs*{
			\int_0^1 \dual{f'(x + t (y -x)) - f'(x)}{y - x}_H \, \d t
		}
		\\&
		\le
		\int_0^1 \frac1t \abs[\big]{
			\dual[\big]{f'(x + t (y -x)) - f'(x)}{t(y - x)}_H
		} \, \d t
		\\&
		\le
		\int_0^1 \frac1t
		L
		\norm{t(y - x)}_H^2
		\d t
		=
		\frac L2 \norm{y - x}_H^2
		.
	\end{align*}

	In order to check ``\ref{lem:lipschitz_derivative:3}$\Rightarrow$\ref{lem:lipschitz_derivative:1}'',
	we take an arbitrary $\rho > 0$ and set
	\begin{equation}\label{eq_O_rho}
		O_\rho :=
		\set{
			x \in O
			\given
			\forall h \in H, \norm{h}_H \le \rho : x +  h \in O
		}.
	\end{equation}
	It is clear that $O_\rho$ is again convex.
	Let us choose some $x,y \in O_\rho$
	with $\norm{y - x}_H \le \rho$
	and $d \in H$ with $y+d, x-d \in O$.
	From \eqref{eq:taylor},
	we get the inequalities
	\begin{align*}
		f(x-d) - f(y) - \dual{f'(y)}{x-d-y}_H &\le \frac L2 \norm{y-x+d}_H^2,\\
		f(y+d) - f(x) - \dual{f'(x)}{y+d-x}_H &\le \frac L2 \norm{y-x+d}_H^2,\\
		-\parens[\big]{f(y+d) - f(y) - \dual{f'(y)}{\mrep[r]{d}{-d}}_H} &\le \frac L2 \norm{d}_H^2,\\
		-\parens[\big]{f(x-d) - f(x) - \dual{f'(x)}{-d}_H} &\le \frac L2 \norm{d}_H^2.
	\end{align*}
	Adding these inequalities leads to
	\begin{equation*}
		\dual{f'(y) - f'(x)}{y-x+2 d}_H
		\le
		L \norm{y-x+d}_H^2 + L \norm{d}_H^2.
	\end{equation*}
	Next, we specialize to
	$d = (x - y + g) / 2$ for some $g \in H$ with $\norm{g}_H \le \rho$.
	Note that
	$\norm{d}_H \le (\norm{y-x}_H + \norm{g}_H)/2 \le \rho$.
	Together with $x,y \in O_\rho$ we find
	$x - d, y + d \in O$.
	For arbitrary $g \in H$ with $\norm{g}_H \le \rho$, this leads to
	\begin{equation}
		\label{eq:ineq_with_g}
		\dual{f'(y) - f'(x)}{g}_H
		\le
		\frac L4 \norm{y - x + g}_H^2 + \frac L4 \norm{x - y + g}_H^2
		=
		\frac L2 \norm{y - x}_H^2 + \frac L2 \norm{g}_H^2
		,
	\end{equation}
	where we used the parallelogram identity.
	We denote by $G \in H$ the Riesz representative of $f'(y) - f'(x) \in H\dualspace$.
	In case
	$\norm{G}_H = \norm{f'(y) - f'(x)}_{H\dualspace} \ge \rho L$,
	we take $g = \rho G / \norm{G}_H$
	in \eqref{eq:ineq_with_g}
	and obtain
	\begin{align*}
		\rho \norm{f'(y) - f'(x)}_{H\dualspace}
		=
		\dual{f'(y) - f'(x)}{g}_H
		\le
		\frac L2 \norm{y - x}_H^2 + \frac L2 \norm{g}_H^2
		\le
		L \rho^2
		.
	\end{align*}
	This yields $\norm{G}_H \le \rho L$.
	Thus, we can insert $g = G / L$
	(in case $L > 0$)
	in \eqref{eq:ineq_with_g}
	and obtain
	\begin{equation*}
		\frac 1 L \norm{f'(y) - f'(x)}^2_{H\dualspace}
		\le
		\frac L2 \norm{y - x}_H^2 + \frac 1 {2 L} \norm{f'(y) - f'(x)}_{H\dualspace}^2
		.
	\end{equation*}
	Altogether,
	this shows
	\begin{equation*}
		\norm{f'(y) - f'(x)}_{H\dualspace}
		\le
		L \norm{y - x}_H
		\qquad
		\forall x,y \in O_\rho, \norm{y - x}_H \le \rho
		.
	\end{equation*}
	For arbitrary $x,y \in O_\rho$,
	we can choose $n \in \N$, $n \ge \norm{y - x}_H / \rho$
	and set
	$x_i := x + i (y - x) / n$, $i = 0,\ldots, n$.
	Due to $\norm{x_i - x_{i-1}}_H = \norm{y - x}_H / n \le \rho$,
	we can use the above estimate to achieve
	\begin{align*}
		\norm{f'(y) - f'(x)}_{H\dualspace}
		&=
		\norm{f'(x_n) - f'(x_0)}_{H\dualspace}
		\le
		\sum_{i = 1}^n \norm{f'(x_i) - f'(x_{i-1})}_{H\dualspace}
		\\
		&\le
		\sum_{i = 1}^n L \norm{x_i - x_{i-1}}_H
		=
		\sum_{i = 1}^n L \frac{\norm{y - x}_H}{n}
		= L \norm{y - x}_H
		\qquad\forall x,y \in O_\rho
		.
	\end{align*}
	Finally,
	$O = \bigcup_{\rho > 0} O_\rho$
	yields
	\eqref{eq:L_lipschitz}.
\end{proof}

The implications
``%
\ref{lem:lipschitz_derivative:1}$\Rightarrow$%
\ref{lem:lipschitz_derivative:2}$\Rightarrow$%
\ref{lem:lipschitz_derivative:3}%
''
remain to hold in the Banach space setting.
However, in the proof of
``%
\ref{lem:lipschitz_derivative:3}$\Rightarrow$%
\ref{lem:lipschitz_derivative:1}%
''%
, we have utilized the parallelogram identity in \eqref{eq:ineq_with_g}.
Thus, the proof does not generalize to Banach spaces.
Instead, we could employ the triangle inequality,
which leads to
\begin{equation*}
	\frac L4 \norm{y - x + g}_H^2 + \frac L4 \norm{x - y + g}_H^2
	\le
	L \norm{y - x}_H^2 + L \norm{g}_H^2
	.
\end{equation*}
By adapting the remaining part of the proof,
we still arrive at \eqref{eq:L_lipschitz}
but with $2 L$ instead of $L$.

By means of an example, we demonstrate that the assertion of \cref{lem:lipschitz_derivative}
indeed fails in Banach spaces.
\begin{example}
	\label{ex:banach_lipschitz}
	We choose $X = (\R^2, \norm{\cdot}_\infty)$,
	thus, $X\dualspace = (\R^2, \norm{\cdot}_1)$.
	We define $f \colon \R^2 \to \R$ via
	\begin{equation*}
		f(x_1, x_2) := \frac12 (x_1^2 - x_2^2)
	\end{equation*}
	Thus,
	\begin{equation*}
		\abs{\dual{ f'(y) - f'(x)}{ y - x}_{\R^2}}
		=  \abs{(y_1 - x_1)^2 - (y_2 - x_2)^2}
		\le
		\norm{ y - x }_\infty^2
	\end{equation*}
	for all $x,y \in \R^2$,
	i.e., \eqref{eq:two_sided_lipschitz}
	is satisfied with $L = 1$.
	However,
	\begin{equation*}
		\norm{ f'(x) - f'(0) }_1
		= \abs{x_1} + \abs{-x_2}
		\le \hat L \norm{ x - 0 }_\infty
		\quad\forall x \in \R^2
	\end{equation*}
	only holds
	for $\hat L \ge 2$.
	Thus, $f' \colon X \to X\dualspace$ is only Lipschitz continuous with constant $2$.
\end{example}

Further, we need a characterization of cocoercive operators
on Hilbert spaces.
\begin{lemma}[\texorpdfstring{\cite[Proposition~4.2]{BauschkeCombettes2011}}{[Bauschke and Combettes, 2011, Proposition 4.2]}]
	\label{lem:cocoercive}
	Let $O \subset H$ be a subset of the Hilbert space $H$.
	Then, for $T \colon O \to H\dualspace$ and $L > 0$
	the following are equivalent.
	\begin{enumerate}
		\item
			$T$ is $1/L$-cocoercive, i.e.,
			\begin{equation*}
				\dual{T(y) - T(x)}{y - x}_H
				\ge
				\frac1L \norm{T(y) - T(x)}_{H\dualspace}^2
				\qquad\forall x,y \in O.
			\end{equation*}
		\item
			$2 T/L - R$ is nonexpansive,
			i.e.,
			\begin{equation*}
				\norm{
					2 T(y)/L - R(y)
					-
					\parens*{
						2 T(x)/L - R(x)
					}
				}_{H\dualspace}
				\le
				\norm{y - x}_H
				\qquad\forall x,y \in O.
			\end{equation*}
	\end{enumerate}
	Here, $R \colon H \to H\dualspace$ is the Riesz isomorphism of $H$.
\end{lemma}
We note that this result follows from some simple and straightforward calculations.

As a last prerequisite,
we show that the so-called strong smoothness of $f$ implies
Gâteaux differentiability.
\begin{lemma}
	\label{lem:smooth_implies_differentiable}
	Let $O \subset X$ be an open and convex subset of the Banach space $X$.
	Suppose that $f \colon O \to \R$ is convex, lower semicontinuous
	and strongly smooth, i.e.,
	\begin{equation*}
		f\parens[\big]{ \lambda x + (1-\lambda) y}
		+
		\frac{L}{2} \lambda (1-\lambda) \norm{y - x}_X^2
		\ge
		\lambda f(x) + (1-\lambda) f(y)
		\quad\forall x,y \in O, \lambda \in (0,1)
		.
	\end{equation*}
	holds for some $L > 0$.
	Then, $f$ is Gâteaux differentiable on $O$.
\end{lemma}
\begin{proof}
	For $x \in O$, $h \in X$ and $t > 0$ small enough,
	we apply the smoothness inequality to $x \pm t h$
	and $\lambda = 1/2$.
	This yields
	\begin{equation*}
		f(x) + \frac{L}{2} \norm{t h}_X^2
		\ge
		\frac12 f(x + t h) + \frac12 f(x - t h).
	\end{equation*}
	Sorting terms and dividing by $t/2$ gives
	\begin{equation*}
		0
		\ge
		\lim_{t \searrow 0}
		\parens*{
			\frac{f(x + t h) - f(x)}{t} + \frac{f(x - t h) - f(x)}{t}
			-
			L t \norm{h}_X^2
		}
		=
		f'(x; h) + f'(x; -h)
		.
	\end{equation*}
	Recall that the existence of the directional derivatives
	follows from $x \in O = \interior(\dom f)$,
	see \cite[Theorem~2.1.13]{Zalinescu2002}.
	This result also gives the sublinearity $f'(x; h) + f'(x; -h) \ge 0$,
	thus $f'(x; \cdot)$ is linear.
	Since $f$ is locally Lipschitz continuous on $O$ by
	\cite[Theorems~2.2.11, 2.2.20]{Zalinescu2002},
	the functional $f'(x; \cdot)$
	is Lipschitz continuous as well.
	Thus, $f'(x; \cdot) \in X\dualspace$
	and this shows the Gâteaux differentiability of $f$.
\end{proof}

Now we are in position to prove the main result.
\begin{theorem}[Baillon--Haddad theorem]
	\label{thm:main}
	Let $O \subset H$ be an open and convex subset of the Hilbert space $H$.
	Suppose that $f \colon O \to \R$
	is convex.
	Then, for $L > 0$, the following are equivalent.
	\begin{enumerate}
		\item
			\label{thm:main:1}
			$f$ is (Gâteaux) differentiable
			and $f'$ is $L$-Lipschitz.
		\item
			\label{thm:main:2}
			$\frac L2 \norm{\cdot}_H^2 - f$ is convex
			and $f$ is lower semicontinuous.
		\item
			\label{thm:main:3}
			$f$ is (Gâteaux) differentiable
			and $f'$ is $1/L$-cocoercive.
	\end{enumerate}
\end{theorem}
\begin{proof}
	Since $H$ is assumed to be a Hilbert space,
	it is straightforward to check that
	\ref{thm:main:2}
	implies the strong smoothness of $f$,
	thus
	$f$ is also Gâteaux differentiable in case
	\ref{thm:main:2},
	see
	\cref{lem:smooth_implies_differentiable}.

	We set $h(x) := \frac L2 \norm{x}_H^2 - f(x)$ for $x \in O$.
	For the auxiliary function $h$
	we have $h'(x) = L R(x) - f'(x)$.
	This directly yields
	\begin{align*}
		\ref{thm:main:1}
		&\quad\Leftrightarrow\quad
		\dual{f'(y) - f'(x)}{y - x}_H \le L \norm{y - x}_H^2
		\quad\forall x,y \in O
		\\
		&\quad\Leftrightarrow\quad
		\dual{h'(y) - h'(x)}{y - x}_H \ge \mrep{0}{L \norm{y - x}_H^2}
		\quad\forall x,y \in O
		\\
		&\quad\Leftrightarrow\quad
		\ref{thm:main:2}
	\end{align*}
	For the other equivalence,
	we employ \cref{lem:cocoercive}.
	To this end, we apply \cref{lem:lipschitz_derivative}
	to the function $g \colon O \to \R$
	defined via $g(x) := 2 f(x) / L - \frac12 \norm{x}_H^2$.
	Note that $g'(x) = 2 f'(x) / L - R$.
	This yields
	\begin{align*}
		\ref{thm:main:1}
		&\quad\Leftrightarrow\quad
		\mrep[r]{0}{-\norm{y - x}_H^2} \le \dual{f'(y) - f'(x)}{y - x}_H \le L \norm{y - x}_H^2
		\quad\forall x,y \in O
		\\
		&\quad\Leftrightarrow\quad
		-\norm{y - x}_H^2 \le \dual{\mrep[r]{g}{f}'(y) - \mrep[r]{g}{f}'(x)}{y - x}_H \le \norm{y - x}_H^2\mrep{}{L}
		\quad\forall x,y \in O
		\\
		&\quad\Leftrightarrow\quad
		\text{$g'$ is nonexpansive}
		\quad\Leftrightarrow\quad
		\ref{thm:main:3}
		.
		\qedhere
	\end{align*}
\end{proof}
The above prove is an adaption of the proof of
\cite[Theorem~3.3]{BauschkeCombettes2009}
to the situation without second-order differentiability
and, thus, gives a simple answer to \cite[Remark~3.5]{BauschkeCombettes2009}.
It is currently not clear whether the equivalence
of \ref{thm:main:1} and \ref{thm:main:3} remains to hold
if $H$ is only assumed to be a Banach space.
Note that both main ingredients of the proof (\cref{lem:lipschitz_derivative,lem:cocoercive})
cannot be transferred directly to Banach spaces, see also the discussion in the next section.
The next example shows
that the implication ``\ref{thm:main:1}$\Rightarrow$\ref{thm:main:2}'' fails in Banach spaces.
\begin{example}
	\label{ex:banach_lipschitz_2}
	We choose $X = (\R^2, \norm{\cdot}_\infty)$,
	thus, $X\dualspace = (\R^2, \norm{\cdot}_1)$.
	We define $f \colon \R^2 \to \R$ via
	\begin{equation*}
		f(x_1, x_2) := \frac12 (x_1^2 + x_2^2)
	\end{equation*}
	Thus,
	\begin{equation*}
		\norm{f'(y) - f'(x)}_1
		=  \abs{y_1 - x_1} + \abs{y_2 - x_2}
		\le
		2\norm{ y - x }_\infty
	\end{equation*}
	for all $x,y \in \R^2$.
	Moreover,
	this inequality is satisfied with equality if $x,y\in \operatorname{span}\set{(1,1)}$.
	Hence, \ref{thm:main:1} is satisfied if and only if $L \ge 2$.

	Define $h(x):=\frac L2 \norm{x}_X^2 - f(x) = \max(x_1^2,x_2^2)- \frac12 (x_1^2 + x_2^2)$.
	Then $h( (1,0)) >0$ but $h( (1,\pm1))=0$. Consequently, $h$ is not convex, and  \ref{thm:main:2} is violated.
	In fact, $h$ is not convex for all $L>0$.
\end{example}
However,
in Hilbert spaces,
assertion
\ref{thm:main:2}
is equivalent to the
strong smoothness of $f$.
Then, the equivalence between Lipschitzness of $f'$
and strong smoothness of $f$
continues to hold in Banach spaces,
see \cref{thm:smoothnes} below.

\section{Convex functions on open, convex subsets of Banach spaces}
\label{sec:gen_and_lim}
In this section, we address generalizations and limitations
of \cref{thm:main}.
In particular, we are interested in convex functions defined on an open subset of a Banach space.
We investigate, which of the claims of \cref{thm:main} and conditions well-known to be equivalent to $L$-Lipschitz continuity of
the derivative remain true in this general situation.

\begin{theorem}
	\label{thm:smoothnes}
	For a convex and lower semicontinuous function
	$f \colon O \to \R$,
	where $O$ is an open and convex subset of a Banach space $X$,
	we consider the following assertions
	with some fixed $L > 0$.
	\begin{enumerate}
		\item
			\label{item:smoothnes:1}
			The function $f$ is strongly smooth
			\begin{equation*}
				f\parens[\big]{ \lambda x + (1-\lambda) y}
				+
				\frac{L}{2} \lambda (1-\lambda) \norm{y - x}_X^2
				\ge
				\lambda f(x) + (1-\lambda) f(y)
				\quad\forall x,y \in O, \lambda \in (0,1)
				.
			\end{equation*}
		\item
			\label{item:smoothnes:2}
			The descent lemma holds
			\begin{equation*}
				f(y) \le f(x) + \dual{x\dualspace}{y - x}_X + \frac L2 \norm{y - x}_X^2
				\qquad\forall x,y \in O, x\dualspace\in \partial f(x)
				.
			\end{equation*}
		\item
			\label{item:smoothnes:6}
			We have
			\begin{equation*}
				\dual{y\dualspace - x\dualspace}{y - x}_X \le L \norm{y - x}_{X}^2
				\qquad \forall (x,x\dualspace), (y,y\dualspace) \in \graph \partial f
				.
			\end{equation*}
		\item
			\label{item:smoothnes:5}
			The subdifferential is Lipschitz continuous
			(thus single-valued)
			\begin{equation*}
				\norm{y\dualspace - x\dualspace}_{X\dualspace} \le L \norm{y - x}_X
				\qquad \forall (x,x\dualspace), (y,y\dualspace) \in \graph \partial f
				.
			\end{equation*}
		\item
			\label{item:smoothnes:4}
			The subdifferential is cocoercive
			\begin{equation*}
				\dual{y\dualspace - x\dualspace}{y - x}_X \ge \frac1L \norm{y\dualspace - x\dualspace}_{X\dualspace}^2
				\qquad \forall (x,x\dualspace), (y,y\dualspace) \in \graph \partial f
				.
			\end{equation*}
		\item
			\label{item:smoothnes:3}
			We have
			\begin{equation*}
				f(y) \ge f(x) + \dual{x\dualspace}{y - x}_X + \frac 1{2 L} \norm{y\dualspace - x\dualspace}_{X\dualspace}^2
				\qquad \forall (x,x\dualspace), (y,y\dualspace) \in \graph \partial f
				.
			\end{equation*}
	\end{enumerate}
	Any of these conditions imply the Gâteaux differentiability of $f$ on $O$.
	Moreover, the following relations hold.
	\begin{enumerate}[label=\textup{(\alph*)}]
		\item
			\label{item:relations_1}
			We have
			``\ref{item:smoothnes:1}$\Leftrightarrow$%
			\ref{item:smoothnes:2}$\Leftrightarrow$%
			\ref{item:smoothnes:6}$\Leftrightarrow$%
			\ref{item:smoothnes:5}''
			and
			``\ref{item:smoothnes:3}$\Rightarrow$%
			\ref{item:smoothnes:4}$\Rightarrow$%
			\ref{item:smoothnes:1}''.

		\item
			\label{item:relations_2}
			In case that $X$ is a Hilbert space,
			we have additionally
			``\ref{item:smoothnes:5}$\Leftrightarrow$\ref{item:smoothnes:4}''.
			This results in
			``\ref{item:smoothnes:1}$\Leftrightarrow$%
			\ref{item:smoothnes:2}$\Leftrightarrow$%
			\ref{item:smoothnes:6}$\Leftrightarrow$%
			\ref{item:smoothnes:5}$\Leftrightarrow$%
			\ref{item:smoothnes:4}''
			and these conditions are implied by
			\ref{item:smoothnes:3}.
		\item
			\label{item:relations_3}
			In case that $O = X$, all assertions are equivalent.
	\end{enumerate}
\end{theorem}
\begin{proof}
	The case \ref{item:relations_3}
	follows from
	\cite[Corollary~3.5.7 and Remark~3.5.2]{Zalinescu2002}
	and case \ref{item:relations_2}
	is implied by \cref{thm:main}.
	\Cref{lem:smooth_implies_differentiable}
	shows that
	\ref{item:smoothnes:1}
	implies the Gâteaux differentiability of $f$.
	Thus, it remains to check the implications from \ref{item:relations_1}.
	For later reuse we recall that $\partial f(x) \ne \emptyset$
	for all $x \in O$,
	since $f$ is continuous on $O$,
	see \cite[Theorems~2.2.20 and 2.4.9]{Zalinescu2002}.

	``\ref{item:smoothnes:1}$\Rightarrow$\ref{item:smoothnes:2}'':
	We already know that $f$ is Gâteaux differentiable.
	Moreover, \ref{item:smoothnes:1} gives
	\begin{equation*}
		\frac{f(y + \lambda (x - y)) - f(y)}{\lambda}
		+
		\frac L2 (1- \lambda) \norm{y - x}_X^2
		\ge
		f(x) - f(y).
	\end{equation*}
	for arbitrary $\lambda \in (0,1)$
	and $\lambda \searrow 0$ results in
	\begin{equation*}
		\dual{f'(y)}{x - y}_X
		+
		\frac L2 \norm{y - x}_X^2
		\ge
		f(x) - f(y).
	\end{equation*}
	Since $\partial f(y) = \set{f'(y)}$,
	this gives
	\ref{item:smoothnes:2}
	with exchanged roles of $x$ and $y$.

	``\ref{item:smoothnes:2}$\Rightarrow$\ref{item:smoothnes:1}'':
	We set $x_\lambda := \lambda x + (1-\lambda) y \in O$
	and choose an arbitrary
	$x_\lambda\dualspace \in \partial f(x_\lambda)$.
	Thus,
	\[
		\begin{aligned}
			f( x ) &\le f(x_\lambda) + \dual{x_\lambda\dualspace}{x - x_\lambda}_X + \frac L2 \norm{x - x_\lambda}_X^2, \\
			f( y ) &\le f(x_\lambda) + \dual{x_\lambda\dualspace}{y - x_\lambda}_X + \frac L2 \norm{y - x_\lambda}_X^2.
		\end{aligned}
	\]
	We multiply the first inequality by $\lambda$
	and the second one by $(1-\lambda)$.
	Adding the resulting inequalities
	and using
	$x - x_\lambda = (1-\lambda) (x - y)$, $y - x_\lambda = \lambda (y - x)$,
	we get
	\ref{item:smoothnes:1}.

	``\ref{item:smoothnes:2}$\Rightarrow$\ref{item:smoothnes:6}'':
	This follows from adding the inequality
	with $(x,x\dualspace)$ and $(y,y\dualspace)$ exchanged.

	``\ref{item:smoothnes:6}$\Rightarrow$\ref{item:smoothnes:2}'':
	We choose an arbitrary $n \in \N$ and define
	\begin{equation*}
		x_0 := x,
		\quad
		x_n := y,
		\quad
		x_i := x + \frac{i}{n} (y - x),
		\quad
		x_0\dualspace := x\dualspace,
		\quad
		x_n\dualspace := y\dualspace
		.
	\end{equation*}
	Further we choose some arbitrary
	$x_i\dualspace \in \partial f(x_i)$ for $i = 1,\ldots, n-1$.
	Thus,
	\begin{align*}
		&f(y) - f(x) - \dual{x\dualspace}{y - x}_X
		=
		f(x_n) - f(x_0) - \dual{x_0\dualspace}{x_n - x_0}_X
		\\&\qquad
		=
		\sum_{i = 1}^n f(x_i) - f(x_{i-1}) - \dual{x_0\dualspace}{x_i - x_{i-1}}_X
		\\&\qquad
		\le
		\sum_{i = 1}^n \dual{x_i\dualspace - x_0\dualspace}{x_i - x_{i-1}}_X
		=
		\sum_{i = 1}^n \frac{1}{i} \dual{x_i\dualspace - x_0\dualspace}{x_i - x_0}_X
		\\&\qquad
		\le
		L \sum_{i = 1}^n \frac{1}{i} \norm{x_i - x_0}_X^2
		=
		L \sum_{i = 1}^n \frac{i}{n^2} \norm{x_1 - x_0}_X^2
		=
		L \frac{n+1}{2 n} \norm{y - x}_X^2
		.
	\end{align*}
	Now, the claim follows from $n \to \infty$.

	``\ref{item:smoothnes:2}$\Rightarrow$\ref{item:smoothnes:5}'':
	Since \ref{item:smoothnes:2} implies \ref{item:smoothnes:1},
	$f$ is differentiable on $O$.
	For an arbitrary $x \in O$, there exists $\rho > 0$
	with $B_\rho(x) := \set{y \in X \given \norm{y - x}_X \le \rho } \subset O$.
	Now, let additionally $y \in B_\rho(x)$ with $x \ne y$ and $\varepsilon > 0$
	be given.
	By the Hahn-Banach theorem we get $z \in X$ with $\norm{z}_X = \norm{y - x}_X \le \rho$
	and
	$(1-\varepsilon) \norm{f'(y) - f'(x)}_{X\dualspace} \norm{z}_X \le \dual{f'(y) - f'(x)}{z}_X$.
	Thus, $x + z \in B_\rho(x) \subset O$
	and we have
	\begin{align*}
		&(1-\varepsilon) \norm{f'(y) - f'(x)}_{X\dualspace} \norm{z}_X
		\\
		&\qquad\le 0 + \dual{f'(y) - f'(x)}{z}_X
		\\
		&\qquad\le
		\bracks*{f(x + z) - f(y) - \dual{f'(y)}{x+z-y}_X} + \dual{f'(y) - f'(x)}{z}_X
		\\
		&\qquad=
		\bracks{f(x + z) - f(x) - \dual{f'(x)}{z}_X}
		+
		\bracks{f(x) - f(y) - \dual{f'(y)}{x - y}_X}
		\\
		&\qquad\le
		\frac L2 \norm{z}_X^2
		+
		\frac L2 \norm{ y - x }_X^2
		=
		L \norm{ y - x }_X^2
		.
	\end{align*}
	Dividing by $\norm{z}_X = \norm{y-x}_X$
	and
	passing to the limit $\varepsilon \searrow 0$ yields
	\begin{equation*}
		\norm{f'(y) - f'(x)}_{X\dualspace}
		\le
		L\norm{y - x}_X
	\end{equation*}
	for all $x,y \in O$
	such that $y \in B_\rho(x) \subset O$
	with $\rho > 0$.

	Now, if $x,y \in O$ are arbitrary,
	there exists $\rho > 0$ with
	$B_\rho(x),B_\rho(y) \subset O$.
	We pick $n \in \N$ with $\norm{y - x}_X / n < \rho$
	and define
	$x_i := x + \frac{i}{n}(y - x)$, $i = 0,\ldots, n$.
	By convexity we have $B_\rho(x_i) \subset O$
	for all $i = 0,\ldots, n$.
	Thus,
	\begin{equation*}
		\norm{f'(y) - f'(x)}_{X\dualspace}
		\le
		\sum_{i = 1}^n \norm{f'(x_i) - f'(x_{i-1})}_{X\dualspace}
		\le
		\sum_{i = 1}^n L \norm{x_i - x_{i-1}}_{X}
		=
		L \norm{y - x}_X
	\end{equation*}
	and this shows \ref{item:smoothnes:5}.

	``\ref{item:smoothnes:3}$\Rightarrow$\ref{item:smoothnes:4}'':
	This follows from adding the inequality
	with $(x,x\dualspace)$ and $(y,y\dualspace)$ exchanged.

	``\ref{item:smoothnes:4}$\Rightarrow$\ref{item:smoothnes:5}''
	and
	``\ref{item:smoothnes:5}$\Rightarrow$\ref{item:smoothnes:6}''
	are straightforward.
\end{proof}
\begin{remark}\leavevmode
	\label{rem:counterexample_and_open_question}
	\begin{enumerate}
		\item
			In the case
			\ref{item:relations_2},
			the missing implication
			``\ref{item:smoothnes:1}$\Rightarrow$\ref{item:smoothnes:3}''
			cannot hold due to the counterexample by \cite[Section~2]{Drori2020},
			which is for the Euclidean case $X=\R^n$.

		\item
			In the general case,
			it is currently not clear whether any of
			the (equivalent) conditions
			\ref{item:smoothnes:1},
			\ref{item:smoothnes:2},
			\ref{item:smoothnes:6},
			\ref{item:smoothnes:5}
			implies condition
			\ref{item:smoothnes:4}.
	\end{enumerate}
\end{remark}

The next lemma shows that the failure of
\ref{item:smoothnes:3}
in case \ref{item:relations_2}
is due to the missing convexity of the range of $f'$.
Since any of the properties in \cref{thm:smoothnes} imply Gâteaux differentiability,
we will formulate all the following results for a Gâteaux differentiable function to simplify the presentation.

\begin{lemma}[\ref{item:smoothnes:4}\texorpdfstring{$\Rightarrow$}{=>}\ref{item:smoothnes:3}]
	\label{lem:on_smoothness_3}
	Let $O \subset X$ be an open and convex subset of the Banach space $X$.
	Suppose that
	the convex, lower semicontinuous and
	Gâteaux differentiable function
	$f \colon O \to \R$
	satisfies
	\cref{thm:smoothnes}~\ref{item:smoothnes:4},
	i.e.,
	$f'$ is cocoercive with constant $L > 0$.
	Further, let $x,y \in O$ be given, such that
	\begin{equation*}
		[f'(x), f'(y)]
		:=
		\set{
			(1-\lambda) f'(x) + \lambda f'(y)
			\given
			\lambda \in [0,1]
		}
		\subset
		f'(O)
		:=
		\set{
			f'(z)
			\given
			z \in O
		}
		.
	\end{equation*}
	Then,
	\begin{equation*}
		f(y) \ge f(x) + \dual{f'(x)}{y - x}_X + \frac 1{2 L} \norm{f'(y) - f'(x)}_{X\dualspace}^2
		.
	\end{equation*}
	That is,
	\cref{thm:smoothnes}~\ref{item:smoothnes:3}
	holds on a subset of $O$.
\end{lemma}
\begin{proof}
	Let $n \in \N$ be arbitrary
	and define
	\begin{equation}
		\label{eq:xip}
		x_i\dualspace :=
		\parens*{
			1 - \frac{i}{n}
		} f'(x)
		+
		\frac{i}{n} f'(y)
		.
	\end{equation}
	Due to our assumption, there exists $x_i \in O$, $i = 1,\ldots, n-1$,
	such that $f'(x_i) = x_i\dualspace$ for all $i = 1, \ldots, n-1$.
	We further set $x_0 := x$ and $x_n := y$.
	Now, we have
	\begin{align*}
		&f(y) - f(x) - \dual{f'(x)}{y - x}_X
		=
		\sum_{i = 1}^n f(x_i) - f(x_{i-1}) - \dual{f'(x_0)}{x_i - x_{i-1}}_X
		\\&\qquad
		\ge
		\sum_{i = 1}^n \dual{f'(x_{i-1}) - f'(x_0)}{x_i - x_{i-1}}_X
		=
		\sum_{i = 1}^n (i-1) \dual{f'(x_i) - f'(x_{i-1})}{x_i - x_{i-1}}_X
		\\&\qquad
		\ge
		\frac1L \sum_{i = 1}^n (i-1) \norm{f'(x_i) - f'(x_{i-1})}_{X\dualspace}^2
		=
		\frac1{L n^2} \sum_{i = 1}^n (i-1) \norm{f'(y) - f'(x)}_{X\dualspace}^2
		\\&\qquad
		=
		\frac{n (n-1)}{2 L n^2} \norm{f'(y) - f'(x)}_{X\dualspace}^2
		.
	\end{align*}
	Passing to the limit $n \to \infty$ yields the claim.
\end{proof}

Under slightly stronger assumptions on $f$ than in the previous lemma,  condition \ref{item:smoothnes:2} implies \ref{item:smoothnes:4}.

\begin{lemma}[\ref{item:smoothnes:2}\texorpdfstring{$\Rightarrow$}{=>}\ref{item:smoothnes:4}]
	\label{hatnochgefehlt}
	Let $O \subset X$ be an open and convex subset of the Banach space $X$.
	Suppose that $f \colon O \to \R$ is convex,  lower semicontinuous, and Gâteaux differentiable.
	Let $L>0$ be given such that \cref{thm:smoothnes}~\ref{item:smoothnes:2}, which is
	\begin{equation*}
		f(y) \le f(x) + \dual{f'(x)}{y - x}_X + \frac L2 \norm{y - x}_X^2
		\qquad\forall x,y \in O
		,
	\end{equation*}
	is satisfied.
	Let $\rho>0$ and $x,y\in O_\rho$ be given with $O_\rho$ defined in \eqref{eq_O_rho}.
	If
	\[
	 [f'(x), f'(y)] \subset f'(O_\rho),
	\]
	then
	\begin{equation}
		\label{eq:yet_another_inequality}
		\dual{f'(y) - f'(x)}{y - x}_X \ge \frac1L \norm{f'(y) - f'(x)}_{X\dualspace}^2
	\end{equation}
	holds,
	which shows that \cref{thm:smoothnes}~\ref{item:smoothnes:4} holds on a subset of $O$.
\end{lemma}
\begin{proof}
First, we consider the case that
$x,y\in O_\rho$ are given
such that
$\norm{f'(y) - f'(x)}_{X\dualspace} < L\rho$.
Take $d\in X$ with $\norm{d}_X<\rho$. Then using the convexity and property \ref{item:smoothnes:2},
we get
\begin{align*}
	\dual{f'(y) - f'(x)}{d}_X
 &= \dual{f'(y)-f'(x)}{y-x}_X + \dual{f'(y)-f'(x)}{d-y+x}_X \\
 &\le \dual{f'(y)-f'(x)}{y-x}_X + f(x+d)-f(y) + f(y-d)-f(x) \\
 & \le  \dual{f'(y)-f'(x)}{y-x}_X + \dual{f'(y)-f'(x)}{-d}_X + L \norm{d}_X^2.
\end{align*}
This implies
\[
 \dual{f'(y)-f'(x)}{d}_X - \frac L2\norm{d}_X^2  \le \frac12\dual{f'(y)-f'(x)}{y-x}_X.
\]
For arbitrary $\varepsilon>0$, we can choose $d \in X$ with $\norm{d}_X = \frac1L\norm{f'(y) - f'(x)}_{X\dualspace}$ such that $\dual{f'(y)-f'(x)}{d}_X \ge \frac1L\norm{f'(y) - f'(x)}_{X\dualspace}^2 -\varepsilon$.
This implies
\[
 \frac 1{2L}\norm{f'(y) - f'(x)}_{X\dualspace}^2 -\varepsilon  \le \frac12\dual{f'(y)-f'(x)}{y-x}_X.
\]
Since $\varepsilon>0$ was arbitrary, this shows \eqref{eq:yet_another_inequality} in case that $\norm{f'(y) - f'(x)}_{X\dualspace}$ is small.

Now, let $x,y \in O_\rho$ be given such that
$[f'(x), f'(x)] \subset f'(O_\rho)$.
We use the same construction as in the proof of \cref{lem:on_smoothness_3}. Let a number $n\in \N$ be given such that
$\norm{f'(y) - f'(x)}_{X\dualspace} < L \rho n$.
Let $x_i\dualspace$ be given as in \eqref{eq:xip}
and choose $x_i \in O_\rho$ with 
$f'(x_i) = x_i\dualspace$ for all $i=1, \dots, n-1$.
We set $x_0=x$ and $x_n:=y$.
By construction, we have $\norm{f'(x_{i+1})-f'(x_i)}_{X\dualspace} = \frac1n \norm{f'(y) - f'(x)}_{X\dualspace} < L\rho$.
Then, the first part of the proof gives
\[
	\frac1L \norm{f'(x_i) - f'(x_{i-1})}_{X\dualspace}^2 \le \dual{f'(x_i)-f'(x_{i-1})}{x_i-x_{i-1}}_X \quad \forall i=1,\dots, n.
\]
By construction, this implies
\[
	\frac 1{Ln^2} \norm{f'(y) - f'(x)}_{X\dualspace}^2 \le \frac1{n}\dual{f'(y)-f'(x)}{x_i-x_{i-1}}_X \quad \forall i=1,\dots, n.
\]
Summation over $i=1,\dots, n$ yields the claim.
\end{proof}

\ifbiber
	\printbibliography
\else
	\bibliographystyle{plainnat}
	\bibliography{references}
\fi

\end{document}